\documentclass[11pt]{article}
\usepackage[utf8]{inputenc}
\usepackage{amsmath}
\usepackage{amssymb}
\usepackage{enumerate} 
\usepackage{amsthm}
\usepackage{hyperref}

\usepackage{todonotes}

\newtheorem{theorem}{\hspace*{\parindent}Theorem}
\newtheorem{lemma}{\hspace*{\parindent}Lemma}

\newtheorem{definition}{\hspace*{\parindent}Definition}

\title{Estimates of  Lebesgue  Constants for Lagrange Interpolation Processes by  Rational Functions under Mild Restrictions to their Fixed Poles} 
\author{S. Kalmykov, A. Lukashov}
\date{\today}

\usepackage{graphicx}

\begin{document}
\maketitle
\begin{abstract}
We estimate the  Lebesgue constants for Lagrange interpolation processes  on one or several intervals by rational functions with fixed poles. We admit that the poles have  accumulation points on the intervals.

To prove it we use an analog of the inverse polynomial image method for rational functions with fixed  poles.
\medskip

\end{abstract}

\section{Introduction}
Interpolation by polynomials on real sets has a long history. Many results in this area can be found in classical books by A.A. Privalov \cite{Privalov}, and  J.~Szabados, P.~Vertesi. \cite{SzabadosVertesi} For basic theory and applications of interpolation processes see also \cite{Mastroianni}. 

Recently several results about the Lebesgue constants of interpolation processes on the finite union of intervals were obtained (see \cite{LukashovSzabados,KrooSzabados}).

In \cite{LukashovSzabados} the authors used an estimate for Lebesgue constants of interpolation processes by rational functions with fixed poles from \cite{Lukashov2005}.

One of the best choices for interpolation by polynomials on an interval is to use zeros of Chebyshev polynomials as nodes of interpolation. Their analogs for rational functions with fixed poles are Chebyshev-Markov rational functions deviated least from zero on an interval. Corresponding interpolation processes were introduced by V. N. Rusak \cite{Rusak1962}. He proved also the estimate $O(\log{n})$ of the related Lebesgue constants under rather general conditions on the fixed poles. Here and below $n$ is the degree of a rational function. Rather general estimate for the related Lebesgue constants is due to A.P.~Starovoitov (\cite{Starovoitov1983} and \cite{Starovoitov}). The last estimate was generalized for the case of two intervals by the second author in \cite{Lukashov1998}. It includes a possible accumulation of fixed poles to the ends of the convex hull of the intervals. The estimate for the case of three or more intervals from \cite{Lukashov2005} did not allow fixed poles to accumulate on the convex hull of the intervals.
An algorithm to compute the interpolation nodes for the case of one interval can be found in \cite{vandeunetal}.

Our first result gives a class of interpolation matrices that allow an estimate of Lebesgue constants of interpolation by rational functions on several intervals with fixed poles having accumulation points at the ends of the convex hull of the intervals.

The first estimate of the Lebesgue constants of interpolation by rational functions on one interval with fixed poles having accumulation points at the ends of the interval was obtained in \cite{Rovba1976}. Later in \cite{Rovba_Dirvuk} an estimate for those Lebesgue constants on one interval was given without any suppositions on the fixed \textit{real} poles. In  \cite{Rovba} the authors proved the estimate $C\log n $ with a constant $C $ which does not depend on the fixed complex poles at $ i a$, $-ia$, where $a\in\mathbb{R}.$

In \cite{Min1998} G.~Min  obtained the sharp estimate $O(\log(n))$ under the same conditions as in \cite{Rusak1962}. 
The following question is a part of an open problem formulated in \cite[Remark~2, p.~129]{Min1998}. 
Whether $\log(n)$ growth of the Lebesgue constant implies that the poles must stay outside an interval which contains $[-1,1]$ in its interior? 
It is clear that the negative answer to this question follows from \cite{Rovba1976,Starovoitov,Rovba}.

In this paper, we improve this proposition. To be more precise, we show that any finite subset of one or even several intervals can be considered as a part of the set of accumulating points of the poles preserving $C\log(n)$ estimate for the Lebesgue constants. 

To prove it we develop an analog of the inverse polynomial image method (see e.g. \cite{Totik} and \cite{KrooSzabados}) for rational functions with fixed poles. 

The main feature of this variant of the approach is the following. We do not need to approximate intervals by inverse polynomial images. Instead of it we may fix a partition of the intervals and move appropriately poles of the rational that provides us the partition. 

The problem concerning infinitely many accumulation points on intervals is still open. 

The structure of the paper is the following. 
First,   notations are introduced, 
and some known, basic auxiliary results are recalled. Then new results are formulated  (Theorems 4-6). Proofs are provided in a separate section. 

\section{Main results}

For any $s\geq 1$ let \begin{equation}
    \label{partition} -1=c_1<c_2<\cdots<c_{2s-1}<c_{2s}=1
\end{equation}  be a finite partition of the interval $[-1,1],$ and let \begin{equation}
\label{intervals}
 J_s=\bigcup_{i=1}^s[c_{2i-1},c_{2i}]   
\end{equation}  be the corresponding set of intervals.

The Lebesgue function of interpolation on $J_s$ by rational functions with fixed poles from the matrix of their reciprocal values \begin{equation}
\label{poles}
\mathfrak{A}=\{a_{k,n}\}\subset \mathbb{D}=\{z\in\mathbb{C}:|z|<1\}    
\end{equation}  for the matrix of interpolation nodes \begin{equation}
\label{nodes}
 \mathfrak{M}=\{x_{k,n},x_{n,n}<x_{n-1,n}<\cdots x_{1,n} \}  
\end{equation} 
is defined as
\begin{equation}
\label{Lebegfunct}
 \lambda_n(\mathfrak{M},\mathfrak{A},x)=\sum_{k=1}^n|\ell_{k,n}(\mathfrak{M},\mathfrak{A},x)|,   
\end{equation} 
where
$$ \ell_{k,n}(\mathfrak{M},\mathfrak{A},x)=\frac{\varpi_n(x)}{\varpi'_n(x_{k,n})(x-x_{k,n})},\quad \varpi_n(x)=\prod_{k=1}^n\frac{x-x_{k,n}}{1-a_{k,n}x}.$$
Furthermore, the Lebesgue constant is given by 
\begin{equation}
    \label{Lebegconst}\lambda_n(\mathfrak{M},\mathfrak{A})=\Vert \lambda_n(\mathfrak{M},\mathfrak{A},x)\Vert_{J_s},
\end{equation}
where $\Vert\cdot\Vert_K $ is the supremum norm on the compact set $K.$

To introduce a regular matrix and describe rational functions such that  their zeros are  appropriate nodes of interpolation we need the following potential theoretic notion. 

\begin{definition} {\rm (see e.g. \cite{Goluzin} or \cite{Ransford})} 
Let $B$ be a domain in $\overline{\mathbb{C}}$ such that its boundary consists of finitely many disjoint piecewise smooth curves or arcs and let $\alpha$ be a closed subset of $\partial B$ also consisting of arcs or curves. We denote by $\omega(z,\alpha,B)$ a bounded harmonic function in $B$ such that $\omega(z,\alpha,B)\rightarrow 1$ as $z$ approaches interior points of $\alpha$, and $\omega(z,\alpha,B)\rightarrow 0$ as $z$ approaches interior points of $\partial B\setminus \alpha$. The function $\omega(z,\alpha, B)$ is called the harmonic measure of $\alpha$ with respect to $B$ at a point $z$.  
\end{definition}

In what follows we use the notation $\omega_j(\cdot)=\omega(1/a_{k,n},[c_{2j-1},c_{2j}],\overline{\mathbb{C}}\setminus J_s)$.

\begin{definition}
Matrix $\mathfrak{A}$ from \eqref{poles} is called regular with respect to $J_s,$ if for every $j=1,\ldots,s$ and $n\geq s$ the sum of the harmonic measures of the interval $[c_{2j-1},c_{2j}]$ with respect to $1/a_{k,n}$ is a positive integer, i.e.
\begin{equation}
    \label{regular}
    \sum_{k=1}^n\omega_j(1/a_{k,n})=q_j, \quad q_j\in\mathbb{N}, \quad j=1,\ldots,s.
\end{equation}
\end{definition}

Regularity of matrices of (reciprocal to) fixed poles is a natural notion, since it guarantees that for every $n\geq s$ $J_s$ can be considered as an inverse image of an interval  under suitable rational function of degree $n$ with (reciprocal to) fixed poles given by $n$-th row of $\mathfrak{A}$. This rational function is a direct generalization to the case of several intervals of the Chebyshev-Markov rational functions. It can be written (up to a constant multiplier) in the form 
\begin{equation}
    \label{ChebMark}\varpi_n(x)=\cos\left(\pi\sum_{k=1}^n\omega(1/a_{k,n},J_s\cap[c_1,x],\overline{\mathbb{C}}\setminus J_s)\right).
\end{equation}
Denote
\begin{equation}
    \label{H}H(x)=\prod_{j=1}^s(x-c_{2j-1})(x-c_{2j}),
\end{equation}
\begin{equation}
    \label{gamma}
    \gamma_n(J_s,\mathfrak{A},x)=\sqrt{-H(x)}\pi\sum_{k=1}^n\omega'(1/a_{k,n},J_s\cap[c_1,x],\overline{\mathbb{C}}\setminus J_s),\quad x\in\,\mathrm{Int}\,J_s.
\end{equation}

Now we can formulate A.P. Starovoitov's theorem and two theorems from earlier papers of  second author in these notations. We enumerate the reciprocal of poles in a row as follows: $a_{1,n}=\ldots=a_{\kappa_n,n}=0$ (here $1/0=\infty),$ $a_{k,n}\neq 0,$ $k=\kappa_n+1,\ldots,n,$ Re $a_{k,n}>0, k=\kappa_n+1,\ldots,n_1$, Re $a_{k,n}<0,k=n_1+1,\ldots,n,$ and for Im $a_{k,n}>0$ Im $a_{k+1,n}=-$ Im
$a_{k,n}.$ Here we do not indicate possible dependence of $n_1$ of $n.$
\begin{theorem}[\cite{Starovoitov}] Let $s=1,\kappa_n\geq 1,n\in\mathbb{N},$ and the matrix of reciprocal values of poles $\mathfrak{A}$ has no accumulation points on $\{|z|=1\}$ with possible exception for points $\{\pm 1\},$ which can be attained by nontangent paths, and  satisfies the conditions
\begin{equation}
    \label{starov1}
    \sum_{k=1}^{n_1}\frac{\sqrt{1-|a_{k,n}|}t}{1-|a_{k,n}|+t^2}\geq C_1 \quad\mathrm{for}\quad\sqrt{1-\max_{1\leq k\leq n_1}|a_{k,n}|}\leq t\leq 1;
\end{equation}
\begin{equation}
    \label{starov2}
     \sum_{k=n_1+1}^{n}\frac{\sqrt{1-|a_{k,n}|}t}{1-|a_{k,n}|+t^2}\geq C_2 \quad\mathrm{for}\quad\sqrt{1-\max_{n_1+1\leq k\leq n}|a_{k,n}|}\leq t\leq 1.
\end{equation}
Then for the matrix $\mathfrak{M}$ such that its $n$-th row consists of zeroes of $\varpi_n$ from~\eqref{ChebMark} the estimate
\begin{equation}
    \label{starovestim} \lambda_n(\mathfrak{M},\mathfrak{A})\leq C_3\log\Vert\gamma_n(J_1,\mathfrak{A},\cdot)\Vert_{J_1}
\end{equation}
holds.

\end{theorem}

\begin{theorem}\cite{Lukashov1998}
Let $s=2$,  the matrix of reciprocal values of poles $\mathfrak{A}$ is regular with respect to $J_2$, has no accumulation points on $\{|z|=1\}$ with possible exception for points $\{\pm 1\},$ which can be attained by non-tangent paths, and  satisfies the conditions \eqref{starov1},\eqref{starov2},
\begin{equation}
    \label{interp2}\min_{x\in J_2}\frac{\gamma_n(J_2,\mathfrak{A},x)}{\sqrt{(x-c_2)(x-c_3)}\gamma_n(J_2,\mathfrak{A},(c_3+1)/2))}\geq C_4, n\geq 2.
\end{equation}
Then for the matrix $\mathfrak{M}$ such that its $n$-th row consists of zeroes of $\varpi$ from~\eqref{ChebMark} the estimate
\begin{equation}
    \label{lukashovestim1} \lambda_n(\mathfrak{M},\mathfrak{A})\leq C_5\log\Vert\gamma_n(J_2,\mathfrak{A},\cdot)\Vert_{J_2}
\end{equation}
holds.
\end{theorem}

\begin{theorem}\cite{Lukashov2005}
Let $s\geq 1,$ the matrix of reciprocal values of poles $\mathfrak{A}\subset\{|z|\leq r\}, 0<r<1,$ is regular with respect to $J_s.$ Then for the matrix $\mathfrak{M}$ such that its $n$-th row consists of zeroes of $\varpi$ from \eqref{ChebMark} the estimate
\begin{equation}
    \label{lukashovestim2} \lambda_n(\mathfrak{M},\mathfrak{A})\leq C_7\log n
\end{equation}
holds.

\end{theorem}

The first our result is the following.
\begin{theorem}
For any $s\geq 2$ and for any system \eqref{intervals} there exists a class of matrices of reciprocal values of poles $\mathfrak{A}$ which are regular with respect to $J_s$, the poles have  accumulation points on  $J_s$, including all points of the partition \eqref{partition} and for the matrix $\mathfrak{M}$ such that its $n$-th row consists of zeroes of $\varpi_n$ from \eqref{ChebMark} the estimate
\begin{equation}
    \label{first} \lambda_n(\mathfrak{M},\mathfrak{A})\leq C_8\log\Vert\gamma_n(J_s,\mathfrak{A},\cdot)\Vert_{J_s}
\end{equation}
holds. 
\end{theorem}

Next theorem shows that the behaviour of the Lebesgue constants can be very unexpected, namely it can be of order $\log n$ in spite of the existence of an arbitrary finite set of accumulation points of poles on the interval.

\begin{theorem}
Let $F$ be an arbitrary finite subset of $[-1,1]$, i.e.  
$$
F=\{t_1,\ldots,t_N\}\subset[-1,1].
$$
Then there exists a matrix of inverse values of poles $\mathfrak{A}$ such that every point of $F$ is an accumulation point of corresponding poles and for  the matrix $\mathfrak{M}$ such that its $n$-th row consists of zeroes of $\varpi$ from \eqref{ChebMark} the estimate \eqref{lukashovestim2} holds.
\end{theorem}

This theorem can be extended to the case of several intervals as follows.

\begin{theorem}
Let $F$ be an arbitrary finite subset of $J_s$, i.e.  
$$
F=\{t_1,\dots, t_N\} \subset J_s.
$$
Then there exists a matrix of inverse values of poles $\mathfrak{A}$ such that every point of $F$ is an accumulation point of corresponding poles and for the matrix $\mathfrak{M}$ such that its n-th row consists of zeroes of $\varpi_n(x)$ from \eqref{ChebMark} the estimate \eqref{lukashovestim2} holds.
\end{theorem}

\section{Proofs}
To prove Theorem 4, we need a lemma relating the fundamental rational functions of Lagrange interpolation under the inverse image of an interval by a rational function with fixed poles.

Let $\mathfrak{Y}$ be an arbitrary matrix of interpolation nodes in $(-1,1),$ and $\mathfrak{B}$ be an arbitrary matrix of reciprocal values of poles in the exterior of $J_1.$ If there exist fixed poles $1/d_k, k=1,\ldots, m,$ such that
$$\sum_{k=1}^m\omega_j(1/d_{k})=p_j, \quad p_j\in\mathbb{N}, \quad j=1,\ldots,s,$$
then the system of intervals $J_s$ is an inverse image of the interval $[-1,1]$ by the rational function
$$r_m(x)=\cos\left(\pi\sum_{k=1}^m\omega(1/d_{k},J_s\cap[c_1,x],\overline{\mathbb{C}}\setminus J_s)\right) $$ with fixed poles $1/d_k, k=1,\ldots, m.$  Hence for any matrix of the interpolation nodes $\mathfrak{Y}\subset(-1,1)$ rows with numbers $mn$ of a matrix of interpolation nodes $\mathfrak{X}\subset$ int $J_s$ are well-defined by $$ \{x_{k,mn}:k=1,\ldots,mn\}=r_m^{-1}(\{y_{i,n}:i=1,\ldots,n\}).$$ We rearrange those rows as $$ y_{k,n}=r_m(x_{(k,j),n}),\, k=1,\ldots,n;\, j=1,\ldots,m. $$

Instead of $\varpi$ for $\mathfrak{X}$ we use notation
$$ \Omega_{mn}(x)=\varpi_n(r_m(x))=\prod_{k=1}^n\frac{r_m(x)-y_{k,n}}{1-b_{k,n}r_m(x)}=C(m,n)\prod_{k=1}^n\prod_{j=1}^m\frac{x-x_{(k,j),n}}{1-a_{(k,j),n}x},$$
where $$ \frac{1}{b_{k,n}}=r_m\left(\frac{1}{a_{(k,j),n}}\right),\, k=1,\ldots,n;\, j=1,\ldots,m.$$
Then the fundamental functions of interpolation for the sets $\mathfrak{Y}$ and $\mathfrak{X}$ are $$ \ell_{k,n}(\mathfrak{Y},\mathfrak{B},y)=\frac{\varpi_n(y)}{\varpi'_n(y_{k,n})(y-y_{k,n})},\, k=1,\ldots,n,$$
and
$$ \ell_{(k,j),n}(\mathfrak{X},\mathfrak{A},x)=\frac{\Omega_{mn}(x)}{\Omega'_{mn}(x_{(k,j),n})(x-x_{(k,j),n})},\, k=1,\ldots,n; j=1,\ldots,m.$$

\begin{lemma}
If  there exist fixed poles $1/d_k, k=1,\ldots, m,$ such that
\begin{equation}
    \label{regularcond}\sum_{k=1}^m\omega_j(1/d_{k})=p_j, \quad p_j\in\mathbb{N}, \quad j=1,\ldots,s,
\end{equation}    
holds then for any  matrix $\mathfrak{Y}$ of interpolation nodes in $(-1,1),$
and for any matrix $\mathfrak{B}$ of reciprocal values of poles in the exterior of $[-1,1]$ we have
\begin{multline*}
 |\ell_{(k,j),n}(\mathfrak{X},\mathfrak{A},x)| 
 \\
 \leq C(d_1,\ldots,d_m)\left(|\ell_{k,n}(\mathfrak{Y},\mathfrak{B},r_m(x))|
 +|\varpi_n(r_m(x))|\left(\left|\frac{\ell_{k,n}(\mathfrak{Y},\mathfrak{B},1)}{\varpi_n(1)}\right| \right.\right.
 \\+\left.\left.\left|\frac{\ell_{k,n}(\mathfrak{Y},\mathfrak{B},-1)}{\varpi_n(-1)}\right|\right)\right),  \ \ \ k=1,\ldots,n;\, j=1,\ldots,m;\, x\in J_s.
 \end{multline*}

\end{lemma}

\begin{proof}[Proof of lemma~1]  Mainly  it is similar to the proof of \cite[Lemma 1]{KrooSzabados} but for the sake of completeness we present it here as well. Below  we will omit second index $n$ and the dependence of constants on $ d_1,\ldots,d_m.$ Put $y=r_m(x)$. 

First, let us suppose $|x-x_{kj}|\leq |r_m'(x_{kj})|. $ Then 
$$\ell_{(k,j)}(\mathfrak{X}_{mn},\mathfrak{A},x)=\frac{\varpi_n(y)}{\varpi_n'(y_{k})}\cdot\frac{1}{r'_m(x_{kj})(x-x_{kj})} $$
and 
\begin{multline*} |y-y_k|=|r_m(x)-r_m(x_{kj})|\leq |r_m'(x_{kj})|\cdot|x-x_{kj}| \\
+\frac{1}{2}\Vert r''_m\Vert_{[-1,1]}(x-x_{kj})^2\leq C_1|r'_m(x_{kj})|\cdot|x-x_{kj}|.
\end{multline*}
Hence $$ |\ell_{(k,j)}(\mathfrak{X}_{mn},\mathfrak{A},x)|\leq C_1\left|\frac{\varpi_n(y)}{\varpi_n'(y_{k})(y-y_k)}\right|=C_1|\ell_{k,n}(\mathfrak{Y}_n,\mathfrak{B},r_m(x))|.
$$

Second, suppose $|x-x_{kj}|> |r_m'(x_{kj})|. $ Then
$$|\ell_{(k,j)}(\mathfrak{X}_{mn},\mathfrak{A},x)|\leq \frac{|\varpi_n(y)|}{|\varpi_n'(y_{k})|}\cdot \frac{1}{\left(r_m'(x_{kj})\right)^2}. $$

Let $\xi_{kj}\in {\rm Int} J_s$  be the nearest point to $x_{kj}$ at the same interval $[c_{2i-1},c_{2i}]$,   $i=1,\ldots,s$, as $x_{kj},$ where $|r_m(\xi_{kj})|=1 $ and $ r'_m(\xi_{kj})=0.$ Denote by $I_{kj}\subset J_s$ the interval with center at $\xi_{kj}$ such that $|r_m''(x)|\geq C_2>0$ for $x\in I_{kj}.$ If such a point $\xi_{kj}$ does not exist, then $I_{kj}=\emptyset.$ To prove that $r_m'(\xi_{kj})\neq 0,$ it is sufficient to differentiate the identity \begin{equation}
\label{abelpell}
  r_m^2(x)-(x^2-1)q_{m-1}^2(x)=1,  
\end{equation}where $q$ are called Chebyshev polynomials of the second kind for rational spaces in \cite[p.~141]{BorweinErdelyi}, and $\sqrt{1-x^2}q$ are called sine Chebyshev-Markov fractions in \cite[p.~49]{Rusak79}. 

Now if $x_{kj}\notin I_{kj}$ then $|r_m'(x_{kj})|\geq C_3>0$ and $$|\ell_{(k,j)}(\mathfrak{X}_{mn},\mathfrak{A},x)|\leq \frac{1}{C_3^2} \frac{|\varpi_n(y)|}{|\varpi_n'(y_{k})|}\leq \frac{2}{C_3^2}|\ell_{k,n}(\mathfrak{Y}_n,\mathfrak{B},r_m(x))|.$$

Finally, let $x_{kj}\in I_{kj}.$ Without loss of generality suppose $r_m(\xi_{kj})=-1.$ Then
$$ y_k+1=r_m(x_{kj})-r_m(\xi_{kj})\leq C_4(\xi_{kj}-x_{kj})^2.$$
So $$ |r_m'(x_{kj})|=|r_m'(x_{kj})-r_m'(\xi_{kj})|=|r_m''(\zeta)|\cdot|x_{kj}-\xi_{kj}|\geq\frac{C_2}{\sqrt{C_4}}\sqrt{1+y_k},$$
$\zeta \in (\xi_{kj},x_{kj})\subset I_{kj}$.

Hence $$|\ell_{(k,j)}(\mathfrak{X}_{mn},\mathfrak{A},x)|\leq C_5\frac{|\varpi_n(y)|}{|\varpi_n'(y_{k})|(1+y_k)}\leq C_6\frac{|\varpi_n(y)|}{|\varpi_n(-1)|}\cdot|\ell_{k,n}(\mathfrak{Y}_n,\mathfrak{B},-1)|. $$ 
\end{proof}

Now we are ready to describe the class of regular matrices of reciprocal values of poles which is mentioned in Theorem 4.

\begin{proof}[Proof of Theorem~4]First, we note that the proof of \cite[Theorem 3]{Lukashov2005} shows   that under supposition $\mathfrak{A}_n\subset\{|z|\leq r\}$, $0<r<1$ and $\mathfrak{A}$ is regular with respect to $J_s$, the estimate
$$ C_7\log n\leq\log\Vert\gamma_n(J_s,\mathfrak{A},\cdot)\Vert_{J_s}\leq C_8\log n$$ holds.

Hence it is sufficient to choose a regular matrix of reciprocal values of poles with $\mathfrak{A}_k\subset\{|z|\leq r\}$, $0<r<1, k\neq mn,$ and for $k=mn$ to choose arbitrary $\mathfrak{B}$ satisfying conditions of Theorem 1, $r_m$ as above, and apply Lemma 1. 
Observe that for this choice of matrices of poles and  nodes we have
$$ \gamma_{mn}(J_s,\mathfrak{A},x)=\frac{\Omega_{mn}'(x)\sqrt{-H(x)}}{\sqrt{1-\Omega^2_{mn}(x)}},$$
$$ \varpi_n'(y)=\frac{\sqrt{1-\varpi_n^2(y)}\gamma_n([-1,1],\mathfrak{B},y)}{\sqrt{1-y^2}},$$
and
$$r_m'(x)=\frac{\sqrt{1-r_m^2(x)}\gamma_m(J_s,\{1/d_1,\ldots,1/d_m\},x)}{\sqrt{-H(x)}}.$$

As the matrix of interpolation nodes $\mathfrak{M}$ we take for $k\neq mn$ zeros of $$\varpi_k(x)=\cos\left(\pi\sum_{j=1}^k\omega(1/a_{j,k},J_s\cap[c_1,x],\overline{\mathbb{C}}\setminus J_s)\right), $$ and for $k=mn$ take nodes which are obtained according to the construction of Lemma 1 like $\mathfrak{X}$ from the zeros of $\varpi_n$ from \eqref{ChebMark} with $s=1.$ \end{proof}

\begin{proof}[Proof of Theorem 5] First, it follows from Theorem~1
 that the estimate $$ \lambda_n(\mathfrak{Y},\mathfrak{B})\leq C_9\log n$$ holds 
for the matrix $\mathfrak{B}$ of inverse values of poles with $b_{1,n}=0,b_{2,n}=\ldots=b_{n,n}=1-1/n $ and the  matrix $\mathfrak{Y}$ such that its $n$-th row consists of zeroes of $\varpi_n$ from \eqref{ChebMark} with $s=1$.

To apply Lemma 1, we need to construct a suitable rational function $r_m(x).$ To do it first of all consider functions
$$f_i(\alpha_1,\ldots,\alpha_N)=\frac{1}{2\pi}\sum_{j=1}^N\arccos\frac{t_i-\alpha_j}{1-\alpha_j t_i},\,i=1,\ldots,N,$$
where $\alpha_1,\ldots,\alpha_N$ are distinct points of $(0,1).$ Direct calculations give $$ \frac{\partial f_i}{\partial \alpha_j}=\frac{\sqrt{1-t_i^2}}{2\pi\sqrt{1-\alpha_j^2}(1-\alpha_jt_i)}, \ \ \,i,j=1,\ldots,N.$$
Hence $$\det\left(\frac{\partial f_i}{\partial \alpha_j}\right)_{i,j=1}^N=\prod_{j=1}^N\frac{\sqrt{1-t_j^2}}{2\pi\alpha_j\sqrt{1-\alpha_j^2}} \cdot\det\left(\frac{1}{\frac{1}{\alpha_i}-t_j}\right)_{i,j=1}^N.$$
The last determinant is the Cauchy determinant, and it implies that the mapping 
\begin{equation*}
(f_1,\ldots,f_N):\{0<\alpha_1<\ldots<\alpha_N<1\}\to(0,N/2)^N
\end{equation*}
is a (locally) diffeomorphism. Hence, there exist $\alpha_1^0,\ldots,\alpha_N^0$ such that 
\begin{equation*}
f_i(\alpha_1^0,\ldots,\alpha_N^0)\in\mathbb{Q}, \ \ i=1,\ldots,N.
\end{equation*} 
Then we can find $k\in \mathbb{N}$ such that \begin{equation*}
k\sum_{j=1}^N\arccos\frac{t_i-\alpha^0_j}{1-\alpha_j^0 t_i}\in 2\pi\mathbb{N},\ \ \ i=1,\ldots,N.
\end{equation*}

Now put 
\begin{equation*}
r_m(x)=\cos \left( k\sum_{j=1}^N\arccos\frac{x-\alpha^0_j}{1-\alpha_j^0 x}\right). 
\end{equation*}
This is a Chebyshev-Markov rational function with poles $1/\alpha_1^0,\ldots,1/\alpha_N^0$ of multiplicity $k.$ 

Now repeat the consideration which was applied to prove Theorem 4. Then we obtain a matrix of interpolation nodes $\mathfrak M$ such that its rows consist of zeros of the Chebyshev-Markov rational functions with poles given by the matrix of inverse values $\mathfrak A$ obtained as follows. For $n\neq \nu m$, $\nu\in \mathbb{N}$, its rows can be taken arbitrarily with only one restriction to give the desired estimate of $\lambda(\mathfrak {M}_n,\mathfrak{A}_n).$ And for $n=\nu m$, $\nu\in \mathbb{N}$, the poles $1/a_{j,n}$ by the construction are the solutions of equations $$r_m(1/a_{j,n})=\frac{1}{1-1/n}.$$ By definition,  $r_m(t_i)=1,i=1,\ldots,N$. It follows from \eqref{abelpell} that $r_m''(t_i)<0$, $i=1,\ldots,N$.  Hence, $$ r_m(x)=1+e_i(x-t_i)^2+o(x-t_i)^2,\,e_i<0,\quad x\to t_i, \ i=1,\ldots,N.$$ Therefore, the points $t_i$, $i=1,\ldots,N$, are accumulation points of the corresponding poles, Q.E.D. \end{proof}

\begin{proof}[Proof of Theorem~6]
First of all we choose poles $1/d_k, k=1,\ldots, m,$  such that \eqref{regularcond} holds (in \cite[Proof of Th.2.2]{LukashovSzabados} was proved that it can be given by choosing $d_k=0, k=1,\ldots,n-s, $ and arbitrarily small other $ d_k$'s). Denote by $ r_m$ the related rational function from Lemma~1 which maps the system of intervals $J_s$ onto $[-1,1].$ Put $y_j=r_m(t_j), j=1,\ldots,N.$ Then apply Theorem~5 to find a matrix of inverse values of poles $\tilde{\mathfrak{A}}=\{\tilde{a}_{j,n}\}$ such that every point $ y_j, j=1,\ldots,N, $ is an accumulation point of corresponding poles and for  the matrix $\mathfrak{M}$ such that its $n$-th row consists of zeroes of $\varpi$ from~\eqref{ChebMark} the estimate \eqref{lukashovestim2} holds. We finish the proof similarly to the end of the proof of Theorem~5, but with choosing $$ r_m(1/a_{i,n})=\tilde{a}_{j,n}.$$ 
\end{proof}

\section*{Acknowledgments}

Alexey Lukashov would like to thank 
Sergei Kalmykov from SJTU 
for the hospitality and support during a visit
to 
Shanghai.

\bibliographystyle{plain}
\bibliography{references}
\medskip

Sergei Kalmykov
\\
 School of Mathematical Sciences, CMA-Shanghai, Shanghai Jiao Tong University, P.R. China 
\\
email address: \href{mailto:sergeykalmykov@inbox.ru}{kalmykovsergei@sjtu.edu.cn}

\medskip{}

Alexey Lukashov
\\
Moscow Institute of Physics and Technology (National Research University), Russia.
\\
email address: \href{mailto:alexey.lukashov@gmail.com}{lukashov.al@phystech.edu}
\end{document}